\documentclass[11pt]{amsart}
\usepackage{anysize}
\marginsize{2.5cm}{2.5cm}{2.5cm}{2.5cm}
\usepackage{amsmath}
\usepackage{amssymb}
\usepackage{amsthm}
\usepackage{color}
\usepackage{hyperref}
\usepackage{enumerate}
\hypersetup{colorlinks=true}

\theoremstyle{plain}
\newtheorem{thm}{Theorem}[section]
\newtheorem{cor}{Corollary}[section]
\newtheorem{lem}[thm]{Lemma}

\theoremstyle{remark}
\newtheorem{qus}{Question}[section]
\newtheorem{rem}{Remark}[section]

\def\Li{\operatorname{Li}}
\def\Res{\operatorname{Res}}

\author{Carlo Sanna}
\address{Department of Mathematics\\Universit{\`a} di Torino\\Via Carlo Alberto 10\\10123 Torino\\Italy}
\email{carlo.sanna.dev@gmail.com}

\author{M{\'a}rton Szikszai}
\address{Institute of Mathematics\\University of Debrecen\\P.O. Box 400.\\H-4002 Debrecen\\Hungary}
\email{szikszai.marton@science.unideb.hu}


\title{A coprimality condition on consecutive values of polynomials}

\begin{document}

\begin{abstract}
Let $f\in\mathbb{Z}[X]$ be quadratic or cubic polynomial.
We prove that there exists an integer $G_f\geq 2$ such that for every integer $k\geq G_f$ one can find infinitely many integers $n\geq 0$ with the property that none of $f(n+1),f(n+2),\dots,f(n+k)$ is coprime to all the others.
This extends previous results on linear polynomials and, in particular, on consecutive integers.
\end{abstract}

\maketitle

\section{Introduction}
Let $s=(s(n))_{n\geq 1}^{\infty}$ be an arbitrary sequence of integers and define $g_s\geq 2$ to be the smallest integer such that one can find $g_s$ consecutive terms of $s$ with the property that none of them is coprime to all the others.
Similarly, let $G_s\geq 2$ denote the smallest integer such that for every $k\geq G_s$ one can find $k$ consecutive terms satisfying the above requirements.
The quantities $g_s$ and $G_s$ may or may not exist. For instance, the sequence of positive even integers has $g_s=G_s=2$, while for the sequence of prime numbers neither exists.
Note that the existence of $G_s$ implies that of $g_s$ and one has $g_s\leq G_s$. 
For less trivial examples see the paper of Hajdu and Szikszai~\cite{hajduszikszai}.

Erd\H{o}s~\cite{erdos} was the first to prove the existence of $G_s$ when $s$ is the sequence of natural numbers.
Later, the combined efforts of Pillai~\cite{pillai} and Brauer~\cite{brauer} gave a more explicit result, namely that $g_s=G_s=17$.
We note that interest in such a problem is twofold.
On one hand, Pillai aimed at the solution of the classical Diophantine problem whether the product of consecutive integers can be a perfect power.
While a complete answer was given by Erd\H{o}s and Selfridge~\cite{erdosselfridge}, Pillai~\cite{pillai2} himself proved, using his already mentioned result from~\cite{pillai}, that it cannot be if one take at most $16$ consecutive terms.
On the other hand, Brauer~\cite{brauer} made connection with his earlier paper~\cite{brauerzeitz} on an old problem, studied already by Legendre~\cite{legendre}, concerning prime gaps.
In fact, Erd\H{o}s~\cite{erdos} himself also studied prime distance of consecutive primes.
Here we would not like to go into further details on any of these relations.

Gradually, the study of $g_s$ and $G_s$ in various sequences, and their importance in analogous problems as the ones mentioned earlier, attracted an increased attention.
Evans~\cite{evans} considered the case when $s$ is an arithmetic progression and proved the existence of $G_s$.
Ohtomo and Tamari~\cite{ohtomotamari} derived the same, but also dealt with numerical aspects by showing that $G_s\leq 384$ for the sequence of odd integers.
The most recent progress is due to Hajdu and Saradha~\cite{hajdusaradha} who gave an effective upper bound on $G_s$ depending only on the difference of the progression together with a heuristic algorithm to find the exact value of it, whenever the number of prime factors of the difference is ``small''.

Observe that both the natural numbers and arithmetic progressions can be considered as consecutive values of linear polynomials.
Recently, Harrington and Jones~\cite{harringtonjones} studied quadratic sequences, that is, for some quadratic $f\in\mathbb{Z}[X]$ one has $s(n)=f(n)$ for every $n\geq 1$.
They computed the exact value of $g_s$ when $f$ is monic or when it belongs to some special families of nonmonic polynomials.
Further, they conjectured that $g_s$ exists and that $g_s\leq 35$ for every quadratic polynomial.
However, they did not consider $G_s$ to any extent.

In this paper, we considerably extend the previous results.
Before stating our result we note that throughout the paper we use the notation $g_f=g_s$ and $G_f=G_s$ and write about consecutive values of the polynomial $f$ instead of consecutive terms of the corresponding sequence $s$.
The main theorem is as follows.

\begin{thm}\label{thm:main}
If $f\in\mathbb{Z}[X]$ is quadratic or cubic, then $G_f$ exists. Further, for every $k\geq G_f$ one can find infinitely many integer $n\geq 0$ such that $f(n+1),f(n+2),\dots,f(n+k)$ has the property that none of them is coprime to all the others.
\end{thm}

Observe that Theorem \ref{thm:main} allows us to immediately settle one part of the conjecture made by Harrington and Jones~\cite{harringtonjones} on $g_f$.

\begin{cor}
\label{cor:harjon}
If $f\in\mathbb{Z}[X]$ is quadratic, then $g_f$ exists.
\end{cor}

Here we do not consider the absolute boundedness of $g_f$, but make some remarks on it instead.
For every positive integer $k\geq 2$, there exists a quadratic polynomial $f\in\mathbb{Z}[X]$ reducible in $\mathbb{Z}[X]$ such that $k\leq g_f\leq G_f$.
This follows easily by taking $d$ to be the product of the first $k$ primes and then looking at the polynomial $f(X)=(1+dX)^2$.
On one hand we have $g_f=g_{1+dX}$ and $G_f=G_{1+dX}$, while on the other we have $k\leq g_{1+dX}\leq G_{1+dX}$.
Neverthless, we could not say anything about the irreducible case and we feel that, despite not stating it anywhere and not excluding reducibles before, Harrington and Jones made their conjecture on this more interesting setting.

Let us finish this section by discussing the main tools we use in the proof of Theorem \ref{thm:main}.
The basic idea is to construct for every quadratic or cubic polynomial $f$ an auxiliary polynomial $\tilde{f}$ that, in some sense, controls the existence of ``close'' solutions to polynomial congruences $f(X)\equiv 0 \pmod{p}$.
Then we show that if $k$ is desirably large, one has enough primes with such close solutions to ``cover'' some block of $k$ consecutive numbers $f(n+1),f(n+2),\dots,f(n+k)$.
The success of this construction relies on the Stickelberger parity theorem, results on the $p$-adic valuations of products of consecutive polynomial values, and lower bounds on the number of certain subsets of primes.

Note that our methods can yield, at least in principle, an effective upper bound on $G_f$.
However, the bound would be too large to be useful in practice.
Further, we emphasize that Theorem~\ref{thm:main} implies the existence of $G_f$ for every quartic polynomial $f\in\mathbb{Z}[x]$ that is reducible in $\mathbb{Z}[X]$ (we always have a factor of degree at most $3$), but our construction already fails to deal with quartic polynomials in general.
We point out this more explicitly in the next section.
Neverthless, the above observations raise two natural questions.

\begin{qus}
Let $f\in\mathbb{Z}[X]$ be of degree at least $4$ and irreducible over $\mathbb{Z}$. Does Theorem~\ref{thm:main} extend to some family of such polynomials?
\end{qus}

\begin{qus}
Does there exist an efficient algorithm that, taken as input a quadratic or cubic polynomial $f\in\mathbb{Z}[x]$, returns $G_f$, or at least a good upper bound for $G_f$?
\end{qus}

\section{Preliminaries}

This section is devoted to the auxiliary results we use in the proof of Theorem \ref{thm:main}.
First, let us fix some notations.
The letter $p$ always denotes a prime number.
For any $x \geq 1$ and for any set of integers $\mathcal{S}$, we put $\mathcal{S}(x) := \mathcal{S} \cap [1, x]$.
We also use the Landau--Bachmann ``Big Oh'' notation $O$ and the associated Vinogradov symbols $\ll$ and $\gg$.
In particular, any dependence of the implied constants is indicated either with subscripts or explicitly stated.
Let
\begin{equation*}
f(X) = a_k X^k + a_{k - 1} X^{k - 1} + \cdots + a_0 ,
\end{equation*}
be a polynomial of degree $k \geq 1$ and with integer coefficients $a_0, \ldots, a_k$.
We define
\begin{equation}\label{equ:ftilde}
\widetilde{f}(X) := a_k^{2k - 2} \prod_{\substack{1 \leq i, j \leq k \\ i \neq j}} (X - (\alpha_i - \alpha_j)) ,
\end{equation}
where $\alpha_1, \ldots, \alpha_k$ are all the roots of $f$ in some algebraic closure.
Observe that $\widetilde{f}$ can be computed from the relation
\begin{equation*}
\Res_X(f(X), f(X + Y)) = a_k^2 Y^k \widetilde{f}(Y) ,
\end{equation*}
where $\Res_X$ is the resultant of polynomials respect to $X$.
In particular, for $k = 2$
\begin{equation}\label{equ:ftilde2}
\widetilde{f}(X) = a_2^2 X^2 - \Delta_f ,
\end{equation}
while for $k = 3$
\begin{equation}\label{equ:ftilde3}
\widetilde{f}(X) = \left(a_3^2 X^2 + 3a_1a_3 - a_2^2\right)^2 X^2 - \Delta_f ,
\end{equation}
where $\Delta_f$ denotes the discriminant of $f$.
We have the following simple, but useful property.

\begin{lem}\label{lem:samegalois}
If $f \in \mathbb{Z}[X]$ is a nonconstant polynomial, then $f$ and $\widetilde{f}$ have the same Galois group over $\mathbb{Q}$.
\end{lem}
\begin{proof}
The identity
\begin{equation*}
\alpha_i = \frac1{k}\left(\sum_{j = 1}^k (\alpha_i - \alpha_j) - \frac{a_{k-1}}{a_k}\right) \quad i = 1, \ldots, k ,
\end{equation*}
implies that $f$ and $\widetilde{f}$ have the same splitting field over $\mathbb{Q}$, and hence the same Galois group.
\end{proof}

The next result deals with another interesting connection between $f$ and $\tilde{f}$, namely it relates $\widetilde{f}$ to ``close'' solutions of the congruence $f(X) \equiv 0 \pmod p$.

\begin{lem}\label{lem:closeroots}
Let $f \in \mathbb{Z}[X]$ be of degree $k=2$ or $3$ and suppose that $p \mid \widetilde{f}(r)$ for some prime number $p \nmid 2a_k$ and some positive integer $r$.
Then there exists an integer $n$ such that 
\begin{equation*}
f(n) \equiv f(n + r) \equiv 0 \pmod p .
\end{equation*}
\end{lem}

\begin{proof}
Let $\alpha_1, \ldots, \alpha_k$ be the roots of $f$ in the algebraic closure of the finite field $\mathbb{F}_p$.
Since $p \mid \widetilde{f}(r)$, by (\ref{equ:ftilde}) we can assume that $\alpha_1 - \alpha_2 = r$, where $r$ is considered as an element of $\mathbb{F}_p$.
If $k = 2$, then from (\ref{equ:ftilde2}) we have that $\Delta_f$ is a square modulo $p$ and, considering $p \nmid 2a_2$, this implies that $\alpha_1, \alpha_2 \in \mathbb{F}_p$ and the claim follows.
If $k = 3$, then by (\ref{equ:ftilde3}) we once again deduce that $\Delta_f$ is a square modulo $p$ and, by the Stickelberger parity theorem~\cite[Theorem~6.68]{berlekamp}, it follows that $f$ has at least one root in $\mathbb{F}_p$.
If $\alpha_1 \in \mathbb{F}_p$ or $\alpha_2 \in \mathbb{F}_p$, then $\alpha_1, \alpha_2 \in \mathbb{F}_p$, and we are done.
If $\alpha_3 \in \mathbb{F}_p$, then $\alpha_1 = 2^{-1}(r -a_1 - \alpha_3) \in \mathbb{F}_p$ and $\alpha_2 = \alpha_1 - r \in \mathbb{F}_p$, and we are done again.
\end{proof}

\begin{rem}
Note that the conclusion of Lemma~\ref{lem:closeroots} is no longer true if the hypothesis on the degree is dropped.
Take for instance, $f(X) = X^4 + 1$. We have that $3 \mid \widetilde{f}(1)$, but the congruence $f(X) \equiv 0 \bmod 3$ has no solutions at all.
\end{rem}

Now for any nonconstant polynomial $f \in \mathbb{Z}[X]$ we define
\begin{equation*}
\mathcal{P}_f := \{p : p \mid f(n) \text{ for some } n \in \mathbb{N}\} .
\end{equation*}
It is well-known that $\mathcal{P}_f$ has a positive relative density $\delta_f$ in the set of prime numbers.
More precisely, the Frobenius density theorem says that $\delta_f = \operatorname{Fix}(\mathcal{G}) / \#\mathcal{G}$, where $\mathcal{G}$ is the Galois group of $f$ over $\mathbb{Q}$, and $\operatorname{Fix}(\mathcal{G})$ is the number of elements of $\mathcal{G}$ which have at least one fixed point, when regarded as permutations of the roots of $f$~(see, e.g., \cite{stevenhagenlenstra}).
We need the following asymptotic formula for $\#\mathcal{P}_f(x)$.

\begin{thm}\label{thm:density}
For any nonconstant polynomial $f \in \mathbb{Z}[X]$, we have
\begin{equation*}
\#\mathcal{P}_f(x) = \delta_f \Li (x) + O_f\!\left(\frac{x}{\exp(C_f \sqrt{\log x})}\right)
\end{equation*}
for all $x \geq 2$, where $\Li$ denotes the logarithmic integral function and $C_f > 0$ is a constant depending on $f$ only.
\end{thm}

\begin{proof}
The formula is a direct consequence of the effective version of the Chebotarev density theorem~\cite[Theorem~3.4]{serre}.
\end{proof}

For each prime number $p$, let $\nu_p$ be the usual $p$-adic valuation.
The next lemma concerns the $p$-adic valuation of products consisting of consecutive values of a polynomial.

\begin{lem}\label{lem:padicQN}
Let $f \in \mathbb{Z}[X]$ be a polynomial without roots in $\mathbb{N}$, and set
\begin{equation}\label{equ:QN}
Q_N := \prod_{n = 1}^N f(n) ,
\end{equation}
for all positive integers $N$.
Then, for any prime number $p$, we have
\begin{equation*}
\nu_p(Q_N) = \frac{t_f N}{p - 1} + O_f\!\left(\frac{\log N}{\log p}\right) ,
\end{equation*}
for all integers $N \geq 2$, where $t_f$ is the number of roots of $f$ in the $p$-adic integers.
\end{lem}
\begin{proof}
This is \cite[Theorem~1.2]{amdeberhanmedinamoll}. 
Note that in \cite{amdeberhanmedinamoll} the error term is written as $O(\log N)$, but looking at the proof one can easily check that it is $O_f(\log N / \log p)$.
\end{proof}

Our last auxiliary result establishes a lower bound for the number of ``big'' prime factors of an irreducible polynomial.

\begin{lem}\label{lem:primdiv}
Let $f \in \mathbb{Z}[X]$ be a nonconstant polynomial.
For each positive integers $N$, let $\mathcal{S}_N$ be the set of all prime numbers $p$ such that $p > N$ and $p \mid f(n)$ for some positive integer $n \leq N$.
Then, we have
\begin{equation*}
\#\mathcal{S}_N \gg_f (1 - \delta_f) N ,
\end{equation*}
for all sufficiently large integers $N$.
\end{lem}
\begin{proof}
We proceed similarly to the first part of the proof of \cite[Theorem~5.1]{evereststevenstamsettward}.

Define $Q_N$ as in (\ref{equ:QN}).
If $f$ has a positive integer root, then the claim follows.
Hence we can assume that $f$ has no roots in $\mathbb{N}$.
In particular, $Q_N \neq 0$ for every integer $N \geq 1$.
Clearly, $\mathcal{S}_N = \{p : p \mid Q_N, \; p > N\}$.
Put $\mathcal{S}_N^\prime := \{p : p \mid Q_N, \; p \leq N\}$, so that
\begin{equation}\label{equ:logQN1}
\log|Q_N| = \sum_{p \in \mathcal{S}_N} \nu_p(Q_N) \log p + \sum_{p \in \mathcal{S}_N^\prime} \nu_p(Q_N) \log p ,
\end{equation}
for every positive integer $N$.
For the rest of the proof, all the implied constants may depend on~$f$.
By Lemma~\ref{lem:padicQN}, we have
\begin{equation*}
\nu_p(Q_N) = \frac{t_f N}{p - 1} + O\!\left(\frac{\log N}{\log p}\right),
\end{equation*}
for every integer $N \geq 2$, and thus
\begin{equation}\label{equ:logQN2}
\sum_{p \in \mathcal{S}_N} \nu_p(Q_N) \log p \ll \sum_{p \in \mathcal{S}_N} \log p \leq \sum_{p \in \mathcal{S}_N} \log |f(N)| \ll \#\mathcal{S}_N \log N .
\end{equation}
Since $\mathcal{S}_N^\prime \subseteq \mathcal{P}_f(N)$, from Theorem~\ref{thm:density} it follows that
\begin{equation*}
\# \mathcal{S}_N^\prime \ll \frac{N}{\log N} ,
\end{equation*}
and that, by partial summation,
\begin{equation*}
\sum_{p \in \mathcal{S}_N^\prime} \frac{\log p}{p - 1} \leq \sum_{p \in \mathcal{P}_f(N)} \frac{\log p}{p - 1} = \delta_f \log N + O\!\left(1\right) ,
\end{equation*}
for every integer $N \geq 2$.
Therefore,
\begin{equation}\label{equ:logQN3}
\sum_{p \in \mathcal{S}_N^\prime} \nu_p(Q_N) \log p \leq \sum_{p \in \mathcal{S}_N^\prime} \left(\frac{kN \log p}{p - 1} + O(\log N)\right) \leq \delta_f k N \log N + O(N) .
\end{equation}
for every integer $N \geq 2$.
Finally, by Stirling's formula
\begin{equation}\label{equ:logQN4}
\log|Q_N| = k N \log N + O(N).
\end{equation}
Putting together (\ref{equ:logQN1}), (\ref{equ:logQN2}), (\ref{equ:logQN3}), and (\ref{equ:logQN4}), we get
\begin{equation*}
\#\mathcal{S}_N \gg (1 - \delta_f) k N + O\!\left(\frac{N}{\log N}\right) ,
\end{equation*}
and the desired result follows.
\end{proof}

\begin{rem}
Note that Lemma~\ref{lem:primdiv} is trivial if $\delta_f = 1$.
\end{rem}

\section{Proof of Theorem~\ref{thm:main}}

Let $f \in \mathbb{Z}[X]$ be a nonconstant polynomial of degree $2$ or $3$.
If $f$ is reducible in $\mathbb{Z}[X]$, then there exists a linear polynomial $h \in \mathbb{Z}[X]$ such that $h(n) \mid f(n)$ for all integers $n$; and the existence of $G_f$ follows immediately from the existence of $G_h$ proved by Evans~\cite{evans}.
Therefore, we can assume that $f$ is irreducible in $\mathbb{Z}[X]$.
Hence the Galois group of $f$ over $\mathbb{Q}$ is precisely one of $S_2$, $S_3$, or $A_3$, and by the Frobenius density theorem $\delta_f$ is $1/2$, $2/3$, or $1/3$, respectively.
Further, by Lemma~\ref{lem:samegalois} we know that $f$ and $\widetilde{f}$ has the same Galois group over $\mathbb{Q}$, and, consequently, by the Frobenius density theorem $\delta_{\widetilde{f}} = \delta_f$.

Let $N$ be a sufficiently large positive integer.
Define $\mathcal{S}_N$ as the set of all prime numbers $p$ such that $p > N/2$ and $p \mid \widetilde{f}(r)$ for some positive integer $r \leq N / 2$.
Thanks to the previous considerations and Lemma~\ref{lem:primdiv}, we have that
\begin{equation}\label{equ:SNc1N}
\#\mathcal{S}_N \geq c_1 N ,
\end{equation}
for all sufficiently large $N$, where $c_1 > 0$ is constant depending only on $f$.
Moreover, Lemma~\ref{lem:closeroots} tell us that for each $p \in \mathcal{S}_N$ there exists two integers $z_p^-$ and $z_p^+$ such that
\begin{equation*}
f(z_p^-) \equiv f(z_p^+) \equiv 0 \bmod p ,
\end{equation*}
and $0 < z_p^+ - z_p^- \leq N / 2 < p$.

Now since
\begin{equation*}
\sum_{p \in \mathcal{P}_f} \frac1{p} = +\infty ,
\end{equation*}
we can fix $s \geq 1$ elements $p_1 < \cdots < p_s$ of $\mathcal{P}_f$ such that
\begin{equation}\label{equ:prodpi}
\prod_{i = 1}^s \left(1 - \frac1{p_i}\right) < \frac{c_1}{3} .
\end{equation}
Moreover, by the definition of $\mathcal{P}_f$, for each $p \in \mathcal{P}_f$ we can pick an integer $z_p$ such that $f(z_p) \equiv 0 \pmod p$.
 
Let $h_1 < \ldots < h_{N_1}$ be all the elements of $\{1, \ldots, N\}$ which are not divisible by any of the primes $p_1, \ldots, p_s$, and let $k_1 < \cdots < k_{N_2}$ be all the remaining elements, so that $N = N_1 + N_2$.
By the Eratosthenes' sieve and (\ref{equ:prodpi}), we have
\begin{equation}\label{equ:N1}
N_1 \leq N \prod_{i = 1}^s \left(1 - \frac1{p_i}\right) + 2^s < \frac{c_1}{2} N ,
\end{equation}
for all sufficiently large $N$.
Let $q_1 < \cdots < q_t$ be all the elements of $\mathcal{S}_N \setminus \{p_1, \ldots, p_s\}$.
From (\ref{equ:SNc1N}) and (\ref{equ:N1}), we get that
\begin{equation*}
t \geq c_1 N - s > \frac{c_1}{2} N > N_1 ,
\end{equation*}
for all sufficiently large $N$.
As a consequence, for any $j = 1, \ldots, N_1$, we can define $r_j = z_{q_j}^-$ if $h_j \leq N/2$, and $r_j = z_{q_j}^+$ if $h_j > N/2$.
Finally, we assume $N$ sufficiently large so that $N \geq 2p_s$.

At this point, note that by construction $p_1, \ldots, p_s$ and $q_1, \ldots, q_{N_1}$ are all pairwise distinct.
Thus, by the Chinese Remainder Theorem, the system of congruences:
\begin{equation*}
\begin{cases}
n \equiv z_{p_i} &\pmod {p_i} \quad i = 1, \ldots, s \\
n \equiv r_j - h_j &\pmod {q_j} \quad j = 1, \ldots, N_1
\end{cases}
\end{equation*}
has infinitely many positive integer solutions.
If $n$ is a solution, then it is easy to see that none of the integers among
\begin{equation*}
f(n + 1), f(n + 2), \ldots, f(n + N)
\end{equation*}
is relatively prime to all the others.

Indeed, take any $h \in \{1, \ldots, N\}$.
On one hand, if $h$ is divisible by some $p_i$, then
\begin{equation*}
f(n + h) \equiv f(n + h \pm p_i) \equiv f(z_{p_i}) \equiv 0 \pmod {p_i} ,
\end{equation*}
so that 
\begin{equation*}
\gcd(f(n + h), f(n + h \pm p_i)) > 1,
\end{equation*}
while $h \pm p_i \in \{1, \ldots, N\}$ for the right choice of the sign, since $N \geq 2p_s$.

On the other hand, if $h$ is not divisible by any of $p_1, \ldots, p_s$, then $h = h_j$ for some $j \in \{1, \ldots, N_1\}$.
If $h_j \leq N / 2$, then
\begin{equation*}
f(n + h) \equiv f(z_{q_j}^-) \equiv 0 \pmod {q_j} ,
\end{equation*}
and
\begin{equation*}
f(n + h + z_{q_j}^+ - z_{q_j}^-) \equiv f(z_{q_j}^+) \equiv 0 \pmod {q_j} ,
\end{equation*}
so that 
\begin{equation*}
\gcd(f(n + h), f(n + h + z_{q_j}^+ - z_{q_j}^-)) > 1 ,
\end{equation*}
while $h + z_{q_j}^+ - z_{q_j}^- \in \{1, \ldots, N\}$.
Similarly, if $h_j > N / 2$ then
\begin{equation*}
\gcd(f(n + h + z_{q_j}^- - z_{q_j}^+), f(n + h)) > 1 ,
\end{equation*}
while $h + z_{q_j}^- - z_{q_j}^+ \in \{1, \ldots, N\}$.

Hence, the existence of $G_f$ has been proved.

\begin{rem}
Note that when $f$ has a linear factor $h=a+dX\in\mathbb{Z}[X]$, we can say more than the existence of $G_f$. Namely, we may apply the results of Hajdu and Saradha~\cite{hajdusaradha} to get an effective upper bound on $G_f$ depending on the number of prime factors of $d$.
\end{rem}

\end{document}